\documentclass[a4paper,12pt]{amsart}

\usepackage{amssymb,amscd,amsthm,mathrsfs}
\usepackage[ansinew]{inputenc}
\usepackage[centertags]{amsmath}
\usepackage{epsfig,graphicx}

\newcommand{\Cph}{C_{\textrm{ph}} }

\newcommand{\Heis}{\mathcal{H}}

 \def\RR{{\mathbb R}}  \def\TT{{\mathbb T}}
  
 \def\ZZ{{\mathbb Z}}

  \def\cH{\mathcal{H}}

\def\cE{\mathcal{E}}    \def\cW{\mathcal{W}}
\def\cF{\mathcal{F}}  \def\cL{\mathcal{L}} \def\cR{\mathcal{R}}

\newcommand{\bbR}{\mathbb{R}}
\newcommand{\bbZ}{\mathbb{Z}}
\newcommand{\bbT}{\mathbb{T}}

\newcommand{\Td}{\mathbb{T}^d}
\newcommand{\Rd}{\mathbb{R}^d}
\newcommand{\Zd}{\mathbb{Z}^d}

\newcommand{\FF}{\mathcal{F}}

\newcommand{\tFF}{\tilde\FF}
\newcommand{\FFcub}{\FF^{cu}_{bran}}
\newcommand{\FFcsb}{\FF^{cs}_{bran}}
\newcommand{\tFFcub}{\tFF^{cu}_{bran}}
\newcommand{\tFFcsb}{\tFF^{cs}_{bran}}

\newcommand{\tM}{\tilde M}

\newcommand{\Es}{E^s}
\newcommand{\Ec}{E^c}
\newcommand{\Eu}{E^u}
\newcommand{\Ecu}{E^{cu}}
\newcommand{\Ecs}{E^{cs}}
\newcommand{\Ws}{\cW^s}

\newcommand{\Wu}{\cW^u}

\newcommand{\Wcs}{\cW^{cs}}

\newcommand{\inv}{^{-1}}

\newtheorem{thm}{Theorem}[section]
\newtheorem{cor}[thm]{Corollary}
\newtheorem{lemma}[thm]{Lemma}
\newtheorem{prop}[thm]{Proposition}

\newtheorem{conjecture}[thm] {\bf Conjecture}

\newtheorem*{teo*}{Theorem}
\newtheorem*{prop*}{Proposition}
\newtheorem*{cor*}{Corollary}
\newtheorem*{goal*}{Goal}

\newtheorem*{af}{Claim}

\newcommand{\bi}{\begin{itemize}}
\newcommand{\ei}{\end{itemize}}

\theoremstyle{definition}

\theoremstyle{remark}
\newtheorem*{remark} {\bf Remark}

\newtheorem*{assumption} {\bf Assumption}
\newtheorem*{notation} {\bf Notation}

\newcommand{\dem}{\vspace{.05in}{\sc\noindent Proof.  }}
\newcommand{\lqqd}{\par\hfill {$\Box$} \vspace*{.05in}}

\newcommand{\eps}{\varepsilon}

\newcommand{\en}{\subset}

\author[A. Hammerlindl]{Andy Hammerlindl}
\address{The University of Sydney, School of Mathematics and Statistics, University of Sydney, NSW 2006, Australia.} \urladdr{www.maths.usyd.edu.au/u/andy/} \email{andy@maths.usyd.edu.au}

\author[R. Potrie]{Rafael Potrie}
\address{CMAT, Facultad de Ciencias, Universidad de la Rep\'ublica, Uruguay}
\urladdr{www.cmat.edu.uy/$\sim$rpotrie}
\email{rpotrie@cmat.edu.uy}

\title[Partial hyperbolicity in 3-nilmanifolds]{Pointwise partial hyperbolicity in 3-dimensional nilmanifolds}
\thanks{A.H.~was partially supported by CNPq (Brazil). R.P.~was partially supported by CSIC group 618, Fondo Clemente Estable FCE-3-2011-1-6749 and the Balzan Research Project of J.Palis.}

\begin{document}
\begin{abstract}
We show the existence of a family of manifolds on which all (pointwise or
absolutely) partially hyperbolic systems are dynamically coherent.
This family is the set of 3-manifolds with nilpotent, non-abelian fundamental
group.
We further classify the partially hyperbolic systems on these manifolds up to
leaf conjugacy.
We also classify those systems on the 3-torus which do not have an attracting
or repelling periodic 2-torus.
These classification results allow us to prove some dynamical consequences,
including existence and uniqueness results for measures of maximal entropy and
quasi-attractors.
\bigskip



\end{abstract}

\maketitle


\section{Introduction}\label{SectionIntroduction}

Since the beginning of its study, there have been two competing candidates
for the definition of partial hyperbolicity.  In both definitions, there is a
continuous splitting of the tangent bundle
\[
    TM = \Es \oplus \Ec \oplus \Eu
\]
invariant under the derivative $Df$ of the diffeomorphism $f:M \to M$
and where the strong expansion of the unstable bundle $\Eu$ and the strong
contraction of the stable bundle $\Es$ dominate any expansion or contraction on
the center $\Ec$.
The distinction between the two definitions is in the exact nature of the
domination.
If (for some Riemannian metric on $M$) the inequalities
\[
    \|Df v^s\| < \|Df v^c\| < \|Df v^u\| \quad \text{and} \quad
    \|Df v^s\| < 1 < \|Df v^u\|
\]
are satisfied for each point $x \in M$ and
%
%
for unit vectors
$v^s \in \Es(x)$,
$v^c \in \Ec(x)$, and
$v^u \in \Eu(x)$,
then $f$ is \emph{pointwise} (or \emph{relatively}) partially hyperbolic.
If $f$ also satisfies the condition that
there are global constants $\lambda,\hat \gamma,\gamma,\mu$
independent of $x$ such that
\[
    \|Df v^s\| < \lambda < \hat \gamma < \|Df v^c\| < \gamma < \mu < \|Df v^u\|
\]
then $f$ is \emph{absolutely} partially hyperbolic. See \cite{HPS} where both notions are discussed.

Only recently have important differences between the two notions come to
light.

In every (pointwise or absolutely) partially hyperbolic system, there are
unique foliations $\Wu$ and $\Ws$ tangent to $\Eu$ and $\Es$ (\cite{HPS}), but there is not
always a foliation tangent to $\Ec$ (see \cite{BuW2}).
A partially hyperbolic diffeomorphism $f$ is \emph{dynamically coherent} if
there are $f$-invariant foliations tangent to $E^{cs}=\Es \oplus \Ec$, $E^{cu}=\Ec \oplus \Eu$,
(and consequently also to $\Ec$).

Brin, Burago, and Ivanov proved that every absolutely partially hyperbolic
system on the 3-torus is dynamically coherent \cite{BBI2}.  Inspired by this
result, Hertz, Hertz, and Ures attempted to extend it to all pointwise
partially hyperbolic systems, but this line of research led them to discover a
counterexample.  There are pointwise partially hyperbolic diffeomorphisms on
$\bbT^3$ for which there is no foliation tangent to the center direction
\cite{HHU}.  Hertz, Hertz, and Ures further asked if there are any
manifolds on which all pointwise partially hyperbolic systems are
dynamically coherent.  In this paper, we answer this question in the
affirmative.

\begin{thm} \label{alldyn}
    Suppose $M$ is a 3-manifold with (virtually) nilpotent fundamental group.
    If $M$ is not finitely covered by $\TT^3$, then every pointwise partially hyperbolic
    diffeomorphism on $M$ is dynamically coherent.
\end{thm}

This builds on results, obtained independently in \cite{HNil} and \cite{Pw},
in the absolute partially hyperbolic setting.

Only certain 3-manifolds can support partially hyperbolic diffeomorphisms.
In particular, if $\pi_1(M)$ is (virtually) nilpotent and $M$ supports partially hyperbolic
diffeomorphisms, then $M$ is (finitely covered by) a circle bundle over a
2-torus \cite{Pw}.
Therefore, Theorem
\ref{alldyn} will follow as a consequence of:

\begin{thm} \label{nildyn}
    If $M$ is a non-trivial circle bundle over $\bbT^2$ and $f:M \to M$ is
    pointwise partially hyperbolic, then $f$ is dynamically coherent and
    there is a unique $f$-invariant foliation tangent to each of the
    bundles.  Further, the bundle projection $\pi:M \to \bbT^2$ may be chosen
    so that every
    leaf of the center foliation is a fiber of the bundle. 
\end{thm}

Beyond dynamical coherence, we also consider the classification problem for
pointwise partially hyperbolic systems.  In studying Anosov flows, the
natural notion of equivalence is topological equivalence.  Two flows
are topologically equivalent if there is a homeomorphism which maps
orbits of one flow to orbits of the other and which preserves the orientations
of the orbits.

The natural notion of equivalence for partially hyperbolic systems is leaf
conjugacy, as introduced in \cite{HPS}.  Two dynamically coherent partially
hyperbolic diffeomorphisms $f$ and $g$ are \emph{leaf conjugate} if there is a
homeomorphism $h$ which maps each center leaf $L$ of $f$ to a center
leaf of $g$ and $h f (L) = g h (L)$.

For every diffeomorphism $f$ of the 3-torus, there is a unique linear map
$A_f:\bbR^3 \to \bbR^3$ such that $A_f(\bbZ^3) = \bbZ^3$ and, viewing $\bbT^3$ as
the quotient $\bbR^3/\bbZ^3$, the quotiented map $A_f:\bbT^3 \to \bbT^3$ is
homotopic to $f$.  We call $A_f$ the \emph{linear part} of $f$.
If $f$ is absolutely partially hyperbolic, then $A_f$ has a partially
hyperbolic splitting and the diffeomorphisms $f$ and $A_f$ are leaf conjugate
\cite{Hammerlindl}.

If $f$ is pointwise partially hyperbolic, it may not have a center foliation (\cite{HHU}),
in which case, it is not possible to define a leaf conjugacy. However, in \cite{Pot} it was shown that 
there is only one obstruction to dynamical coherence in $\TT^3$ and here we show that
this is also the only obstruction to extending the classification.

\begin{thm} \label{torusconj}
    Let $f:\bbT^3 \to \bbT^3$ be pointwise partially hyperbolic. Then, either
    \begin{itemize}
        \item $f$ is not dynamically coherent and there is a periodic 2-torus
        $T = f^k(T)$ tangent either to $\Ec \oplus \Eu$ or $\Ec \oplus \Es$, or
        \item $f$ is dynamically coherent and leaf conjugate to its linear
        part.
    \end{itemize}  \end{thm}
In the first case above, such a torus $T$ is transverse either to $\Eu$ or $\Es$
and is therefore either an attractor or a repeller.  Such phenomena are
impossible when the chain-recurrence set $\cR(f)$ is all of $\bbT^3$ (see \cite[Chapter 10]{BDV}).

\begin{cor}
    If $f:\bbT^3 \to \bbT^3$ is pointwise partially hyperbolic and
    $\cR(f)=\bbT^3$, then $f$ is dynamically coherent and leaf conjugate to
    its linear part.
\end{cor}
Note that if the non-wandering set $\Omega(f)$ is all of $\TT^3$ (for example
when $f$ is transitive or when $f$ is volume-preserving) the previous
corollary applies.  

One can also show that a torus $T$ as in Theorem
\ref{torusconj} cannot exist when the linear part is hyperbolic (see for example Proposition A.1 of \cite{Pot}).

\begin{cor} \label{isoAnosov}
    If $f:\bbT^3 \to \bbT^3$ is pointwise partially hyperbolic and
    its linear part $A_f$ has no eigenvalues of modulus one,
    then $f$ is dynamically coherent
    and leaf conjugate to $A_f$.
\end{cor}

The proof of Theorem \ref{torusconj} combines results
from \cite{Hammerlindl}, where leaf conjugacy was obtained for absolutely
partially hyperbolic systems, and \cite{Pot}, where dynamical coherence was
studied in the
pointwise case. This proof is given in section \ref{SectionTorusConj} after some
preliminaries are introduced in section \ref{SectionBranching}.

For the manifolds considered in Theorem \ref{nildyn}, all of the systems are
dynamically coherent and we can classify all of them.
If $M$ is a circle bundle over $\bbT^2$, it is a \emph{nilmanifold}, i.e. a compact
quotient $G/\Gamma$ of a nilpotent Lie group $G$.  For a diffeomorphism $f:M \to
M$, there is a unique Lie group automorphism on $G$ which descends to a map
$\Phi_f:M \to M$ homotopic to $f$.  Call $\Phi_f$ the \emph{algebraic part} of
$f$.

\begin{thm}\label{nilconj}
    Suppose $M$ is a 3-dimensional nilmanifold, $M  \ne 
    \bbT^3$,
    and $f:M \to M$ is pointwise partially hyperbolic.  Then, $f$ is leaf
    conjugate to its algebraic part $\Phi_f$.
\end{thm}
In 2001, Pujals made a conjecture on transitive partially hyperbolic
diffeomorphisms which we paraphrase here (see \cite{BW}).

\begin{conjecture}
    [Pujals  (2001)]
    Up to a finite cover, every transitive partially hyperbolic diffeomorphism
    on a 3-manifold is leaf conjugate to
    \begin{itemize}
        \item an Anosov diffeomorphism on $\bbT^3$,
        \item a time-one map of an Anosov flow, or
        \item a topological skew product over an Anosov map on $\bbT^2$.
    \end{itemize}  \end{conjecture}

More recently (in 2009), Hertz, Hertz, and Ures conjectured that transitivity could be replaced by dynamical coherence in the previous conjecture. Moreover, they posed the following conjecture which, if proven to be true implies in particular that transitive partially hyperbolic diffeomorphisms are dynamically coherent. 

\begin{conjecture}
    [Hertz, Hertz, Ures (2009)]
    If a partially hyperbolic diffeomorphism $f$ on a 3-manifold is not
    dynamically coherent, there is a periodic torus
    $T = f^k(T)$ tangent either to $\Ec \oplus \Eu$ or $\Ec \oplus \Es$.
\end{conjecture}

Both conjectures were posed in the pointwise case, and the results of this
paper show that both conjectures are true when the manifold in question has
(virtually) nilpotent fundamental group. As we mentioned, this last conjecture
in the case of $\TT^3$ was already established in \cite{Pot}. We mention that
in a forthcoming paper (\cite{HP}) we plan to extend our results to all 3-manifolds with \emph{solvable} fundamental group.

One obvious reason to study pointwise partially hyperbolic systems over
absolutely partially hyperbolic systems is that it is more general.  One
family of diffeomorphisms properly includes the other.
Another important
motivation is the study of robust transitivity and stable ergodicity (see \cite{BDV,WilkinsonSurvey}):
D\'iaz, Pujals, and Ures
proved that every $C^1$ robustly transitive diffeomorphism of a 3-manifold is
partially hyperbolic in the weak sense, that is, it satisfies the definition
given at the start of this paper, with the possible caveat that one of the
bundles $\Es$, $\Ec$, or $\Eu$ may be zero \cite{DPU}.
Their theorem is definitely a pointwise theorem, as there are robustly
transitive diffeomorphisms which are not absolutely partially hyperbolic.
The result shows that pointwise partial hyperbolicity is a notion which arises
naturally when studying the space of $C^1$ diffeomorphisms.

We remark that once the topological classification is obtained, 
further dynamical consequences follow.
Section
\ref{Section-Consequences} explores some more-or-less direct consequences of
our result. First, we study the existence and finiteness of maximal entropy
measures which are a direct application of our main results and previous ones
\cite{HHTU,Ures}. Then, we obtain another dynamical consequence which holds
for partially hyperbolic diffeomorphisms of non-toral nilmanifolds: 

\begin{prop}\label{dynamicalCons1} Let $f: M \to M$ a partially hyperbolic diffeomorphism of a nilmanifold $M \neq \TT^3$. Then, the foliations $\cW^u$ and $\cW^s$ have a unique minimal set (in particular, a unique quasi-attractor\footnote{See section \ref{Section-Consequences} for a definition.}).
\end{prop}

\begin{notation}
    Throughout this paper ``partially hyperbolic'' without further qualifiers
    is taken to mean pointwise partially hyperbolic and
    where all three bundles $\Eu$, $\Ec$, and $\Es$ are non-zero.
\end{notation}

\section{Branching foliations}\label{SectionBranching}

In this section, we introduce the notions of ``almost aligned''
and ``asymptotic'' for branching and non-branching foliations, and
review the results of Burago and Ivanov.

A (\emph{complete}) \emph{surface} in a $3$-manifold $M$
is a $C^1$ immersion $\imath : U \to M$ of a
connected smooth $2$-dimensional manifold without boundary $U$
which is complete with the metric induced by the metric on $M$
pulled back by $\imath$.

A \emph{branching foliation} on $M$ is a collection of complete
surfaces tangent to a given continuous $2$-dimensional
distribution on $M$ such that:

\begin{itemize}
\item Every point is in the image of at least one surface.
\item There are no topological crossings between any two surfaces of the
collection.  That is, no curve lying on one leaf can cross through another
leaf.
\item It is complete in the following sense:
if $x_k \to x$ and $\imath_k$ are surfaces of the
partition having $x_k$ in its image we have that $\imath_k$
converges in the $C^1$-topology to a surface of the collection
with $x$ in its image (see \cite{BI} Lemma 7.1).
\end{itemize}

The image $\imath(U)$ of each surface in a branching foliation is called a
\emph{leaf}.

\begin{thm}
    [Burago-Ivanov \cite{BI}, Theorem 4.1] \label{branfol}
    If $f$ is a partially hyperbolic diffeomorphism on a 3-manifold $M$, such
    that the bundles $E^s$, $E^c$ and $E^u$ are orientable and the orientation
    is preserved by $Df$, then, there is a (not necessarily unique)
    $f$-invariant branching foliation $\FFcsb$ tangent to $\Ecs$.  Further,
    any curve tangent to $\Es$ lies in a single leaf of $\FFcsb$.
\end{thm}
A similar foliation $\FFcub$ is defined tangent to $\Ecu$ under the same
hypotheses.

\begin{remark}
In proving results of dynamical coherence, we will need to show in some cases
that the branching foliation given by this theorem is a true foliation.
A helpful observation is that if $\cF$ is a branching foliation
such that every point in $M$ belongs to a unique leaf, then $\cF$ is indeed a
true foliation (see Proposition 1.6 and Remark 1.10 in \cite{BW}). A
\emph{foliation} for us will mean a $C^{1,0+}$ foliation in the notation of
\cite{CandelConlon}, that is, a $C^0$-foliation with $C^1$-leaves tangent to a
continuous distribution. 
\end{remark}

A surface $\imath: S \to M$ can be lifted to a surface
$\tilde\imath:\tilde S \to \tilde M$ where $\tilde S$ and $\tilde M$ are the
universal covers.
This lift is not unique in general.
For a branching foliation $\FF$ on $M$, consider the collection of all
possible lifts of all surfaces.  This collection defines a unique branching
foliation $\tFF$ called the lift of $\FF$ to $\tM$.
A branching foliation $\FF_1$ is \emph{almost aligned} with a branching
foliation $\FF_2$ if there is $R>0$ such that
each leaf of $\tFF_1$ lies in the $R$-neighborhood of a leaf of $\tFF_2$.
Note that this defines a relation on the set of foliations which is
transitive, but not necessarily symmetric.

Burago and Ivanov further proved that from any branching foliation, a
non-branching foliation can be constructed by deforming each leaf by an
arbitrarily small amount \cite{BI} (Section 7).  In the partially hyperbolic
setting, if the
branching foliation is transverse to the unstable foliation, the new foliation
is also transverse, and by this virtue, it does not contain a Reeb component
\cite{BI} (Section 2).
These results imply the following.
\begin{thm}
    [Burago-Ivanov \cite{BI}] \label{nearfol}
    Under the hypotheses of Theorem \ref{branfol}, there
    is a (non-branching) $C^{1,0+}$ Reebless foliation $\cW$ such that
    $\cW$ is almost aligned with $\FFcsb$ and
    $\FFcsb$ is almost aligned with $\cW$.
    A similar foliation exists for $\FFcub$.
\end{thm}

Note that the new foliation $\cW$ is neither $f$-invariant nor tangent to
$\Ecs$.


In the specific case of the torus $\bbT^d = \bbR^d / \bbZ^d$, we define a
notion of the ``asymptotic'' behavior of a foliation.
A linear subspace $V \subset \bbR^d$ defines a linear foliation $\tFF_V$ where
the leaves are the fibers of the orthogonal projection $\pi:\bbR^d \to V^\perp$.
A branching foliation $\tFF$ is \emph{asymptotic} to $\tFF_V$ if for $\eps>0$
there is $K>0$ such that
\[
    \|x - y\| > K  \quad \Rightarrow \quad  \|\pi x - \pi y\| < \eps \|x - y\|
\]
for all $x,y$ on the same leaf of $\tFF$.  Note that if $\tFF$ is almost aligned
with $\tFF_V$ then $\tFF$ is asymptotic to $\tFF_V$.

As a final definition,
a foliation $\tilde \cW$ on a manifold $\tM$ is \emph{quasi-isometric}
if there is a global constant $Q > 1$ such that
\[
    d_{\tilde \cW}(x,y) < Q\ d_{\tM}(x,y) + Q
\]
for any points $x$ and $y$ on the same leaf of $\tilde \cW$.  Here, $d_{\tilde
\cW}$ denotes the distance inside a leaf of the foliation.

\begin{figure}[t]
\begin{center}
\includegraphics{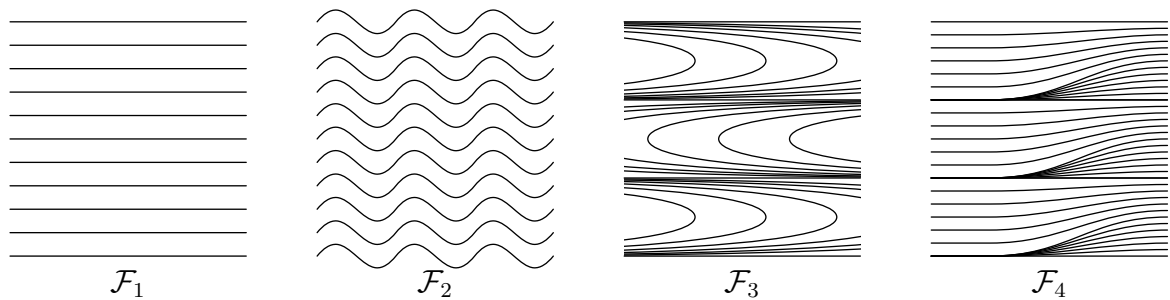}
\end{center}
\caption{A graphical depiction of four examples of branching foliations on the
plane.
Only $\FF_4$ exhibits branching; the others are true foliations.
For any $1\le i,j\le 4$,
one can verify that $\FF_i$ is almost aligned with $\FF_j$.
The foliation $\FF_1$ is linear and the other foliations are asymptotic to
$\FF_1$.
Finally, $\FF_3$ is the only one which is not quasi-isometric.
}
\label{foln-figure}
\end{figure}

Figure \ref{foln-figure} gives a graphical illustration of these definitions.

\pagebreak

\section{Leaf conjugacy on the 3-torus}\label{SectionTorusConj}

The main result of \cite{Hammerlindl} is the following.

\begin{thm} \label{absconj}
    Let $f:\Td \to \Td$ be an absolutely partially hyperbolic diffeomorphism.
    Suppose
    \begin{itemize}
        \item $\tilde {\cW}^\alpha_f$ is a quasi-isometric foliation ($\alpha=u,s$),
        and
        \item $\Ec_f$ is one-dimensional.
    \end{itemize}
    Then, $f$ is leaf conjugate to its linear part.
\end{thm}
While the theorem is stated under the stronger hypothesis of absolute partial
hyperbolicity, many of the arguments of the proof apply equally well in the
pointwise case.
In fact, absolute partial hyperbolicity
is used only in the beginning of one of the sections (Chapter 2, titled ``Nice
Properties of the Invariant Manifolds'').

In this paper's introduction, we used $A_f$ to denote the linear part of $f$,
whereas in \cite{Hammerlindl} it is called the ``linearization'' and denoted
by $g$.
In this section, we use the symbols $A_f$ and $g$ interchangeably depending on
the context.

At the start of the proof in \cite{Hammerlindl}, the constants
$0 < \lambda < \hat \gamma < \gamma < \mu$
associated to the absolutely partially hyperbolic splitting of $f$ are used to
define a splitting for $g$.
Next,
the quasi-isometry assumption is used to show dynamical coherence.  This
follows from a result of Brin which holds only in the absolutely partially
hyperbolic case \cite{Brin}.
Then, \cite{Hammerlindl} states and proves a number of propositions comparing
the foliations of $f$ and $g$.
The proofs of exactly three of these propositions (numbered (2.3), (2.5), and
(2.7)) rely on absolute instead of pointwise partial hyperbolicity.
Using the definitions given in this paper, these three propositions may be
stated as follows.
\begin{enumerate}
    \item[I.]
    $\tilde {\cW}^\alpha_f$ is asymptotic to $E^\alpha_g$ ($\alpha=u,s$).
    \item[II.]
    $\tilde {\cW}^\alpha_f$ is almost aligned with $\tilde {\cW}^\alpha_g$
    ($\alpha=cu,cs$).
    \item[III.]
    For all $x \in \Rd$,
    $
        \tilde {\cW}^{cs}_f(x) \cap \tilde {\cW}^{u}_f(x) = \{x\}  $
    and
    $
        \tilde {\cW}^{cu}_f(x) \cap \tilde {\cW}^{s}_f(x) = \{x\}.
    $  \end{enumerate}
The version of II in \cite{Hammerlindl} also states that $\tilde {\cW}^c_f$ is
almost aligned with $\tilde {\cW}^c_g$, but this easily follows as a consequence
of II as given above.
While the proof of III in \cite{Hammerlindl} relies on absolute partial
hyperbolicity, it can be replaced by a short proof that works in the pointwise
case.

\begin{proof}
    [Proof of III assuming I and II]
    Suppose
    $x  \ne  y \in \tilde {\cW}^{cs}_f(x) \cap \tilde {\cW}^{u}_f(x)$
    and define $v_n = f^n(x) - f^n(y)$.
    Quasi-isometry of the unstable foliation implies that $\|v_n\| \to \infty$
    as $n \to \infty$.
    Then, property I implies that the angle of $v_n$ with $\Eu_g$ goes to zero
    and property II implies that the angle with $\Ecs_g$ goes to zero.  These
    cannot both be true.
\end{proof}
As the rest of the proof of Theorem \ref{absconj} follows assuming only
pointwise partial hyperbolicity, we may reformulate the statement thus.

\begin{thm} \label{pwtd}
    Let $f:\Td \to \Td$ be a pointwise partially hyperbolic diffeomorphism with
    linear part $g:\Td \to \Td$.  Suppose $g$ has a (linear) partially
    hyperbolic splitting and that
    \begin{itemize}
        \item $f$ is dynamically coherent,
        \item $\tilde {\cW}^\alpha_f$ is a quasi-isometric foliation ($\alpha=u,s$),
        \item $E^c_f$ is one-dimensional, and
        \item $f$ and $g$ satisfy properties I and II above.
    \end{itemize}
    Then, $f$ is leaf conjugate to $g$.
\end{thm}
\begin{remark}
    The constants $\lambda$ and $\mu$ from the absolutely partially hyperbolic
    splitting are used throughout the exposition in \cite{Hammerlindl}.
    However, apart from the special cases mentioned above, all we need are
    constants $\Cph > 1$ and $0 < \lambda < 1 < \mu$ satisfying
    \begin{align*}
        \frac{1}{\Cph} \mu^n \|v\| &< \|Df^n v\|&&
        \text{for}\ v \in \Eu_f(x) \setminus \{0\}, \\
        \|Df^n v\| &< \Cph \lambda^n \|v\|&&
        \text{for}\ v \in \Es_f(x) \setminus \{0\}.
    \end{align*}
    Such constants exist for any pointwise partially hyperbolic system.
\end{remark}
\medskip

This finishes the discussion for diffeomorphisms on a torus $\bbT^d$ of
general dimension $d  \ge  3$.
In the specific case of dimension three,
we combine Theorem \ref{pwtd} with the results in \cite{Pot} to prove Theorem \ref{torusconj}.
First, we note that the linear part $g = A_f$ is partially hyperbolic.

\begin{prop}
\label{Proposition-PHinT3isPHthelinearpart}
If $f: \TT^3 \to \TT^3$ is partially hyperbolic, the linear part $A_f$
has real eigenvalues
$\lambda_1,\lambda_2$ and $\lambda_3$ where
\[    |\lambda_1|< |\lambda_2| < |\lambda_3|
    \quad \text{and} \quad
    |\lambda_1|< 1 < |\lambda_3|.  \]
The associated eigenspaces define a partially hyperbolic splitting for
$A_f$.
\end{prop}

\begin{proof}
    Let $\lambda_1,\lambda_2$ and $\lambda_3$ be the (possibly complex)
    eigenvalues of $A_f$, ordered so that
    $
        |\lambda_1|  \le  |\lambda_2|  \le  |\lambda_3|.
    $
    Burago and Ivanov proved that $|\lambda_1| < 1 < |\lambda_3|$; see \cite{BI}
    (Theorem 1.2).   If $|\lambda_2|=1$, we are done, since non-real eigenvalues
    must come in conjugate pairs.  If $|\lambda_2| > 1$, then \cite{Pot}
    (Theorem A) shows that $\lambda_2$ and $\lambda_3$ are real and distinct,
    and $\lambda_1$, as the only eigenvalue with modulus smaller than $1$, must be
    real as well.  Finally, if $|\lambda_2| < 1$, consider $f \inv$ instead.
\end{proof}

To show that the unstable foliation of $f$ is asymptotic to the
unstable foliation of $A_f$, we need a technical lemma.

\begin{lemma} \label{conelemma}
    Suppose
    \begin{itemize}
        \item
        $\FF$ is a one-dimensional foliation on $\bbT^3$,
        \item
        $\FF$ is invariant under a diffeomorphism $f:\bbT^3 \to \bbT^3$,
        \item
        the linear part $A_f$ of $f$ is partially hyperbolic and therefore
        defines a plane $P$ through the origin tangent to $\Ecs_{A_f}$, and
        \item
        there is $K>0$ and a cone $\cE$ transverse to $P$ such that
        \[
            \|y-x\| > K  \quad \Rightarrow \quad  y-x \in \cE
        \]
        for all $x \in \bbR^3$ and $y \in \tilde \FF(x)$.  \end{itemize}
    Then, $\tilde \FF$ is asymptotic to the unstable foliation of $A_f$.
\end{lemma}
Here, we take a cone transverse to $P$ to mean a set of the form
\[
    \cE = \{ v \in \bbR^3 : \|\pi v\| < c \|v\| \}
\]
where $0 < c < 1$ and where $\pi: \bbR^3 \to P$ is any projection.
The proof of Lemma \ref{conelemma} relies on two properties.
First, that as $k \to \infty$, the angle between a vector $v \in A_f^k(\cE)$ and
the one-dimensional subspace $E^u_{A_f}$ tends uniformly to zero.
Second, that for fixed $k$, as $\|y-x\|$ goes to infinity, the angle between
$f^k(y)-f^k(x)$ and $A_f^k(y)-A_f^k(x)$ tends uniformly to zero.
We leave the details to the reader.

Several of the results of \cite{Pot} are used to establish the leaf
conjugacy.
We compile them all into the following statement.

\begin{thm}\label{Teo-Potrie} Let $f: \TT^3 \to \TT^3$ be a
partially hyperbolic diffeomorphism such that there is no
periodic 2-torus tangent to $E^s \oplus E^c$.  Then{:}
\begin{enumerate}
\item There is a unique $f$-invariant foliation $\cW_f^{cs}$ tangent to $\Ecs_f$.
\item ${\tilde \cW}_f^{cs}$ is almost aligned with ${\tilde \cW}_{A_f}^{cs}$.
\item Each leaf of ${\tilde \cW}_f^{cs}$ intersects each leaf of
${\tilde \cW}_f^u$ exactly once.\\(That is, the two foliations have global
product structure.)
\item \label{qiitem} ${\tilde \cW}_f^u$ is a quasi-isometric foliation.
\item ${\tilde \cW}_f^u$ is asymptotic to ${\tilde \cW}_{A_f}^u$.
\end{enumerate}
\end{thm}

\begin{proof}
    The first item is a restatement of \cite{Pot} (Theorem B).  The second and
    third items follow from \cite{Pot} (Proposition A.1) in the case where $A_f$
    is Anosov, or from \cite{Pot} (Proposition 8.7) in the case where it is
    not.
    With these established, \cite{Pot} (Proposition 6.9) implies that ${\tilde
    \cW}_f^u$ is quasi-isometric (iv) and that there is a cone $\cE$ as in
    Lemma \ref{conelemma} above.  The last item then follows from this lemma.
\end{proof}

Note that this result also holds with the roles of the stable and unstable
bundles exchanged.
Theorem \ref{torusconj} now follows easily as a combination of
Theorem \ref{pwtd},
Proposition
\ref{Proposition-PHinT3isPHthelinearpart}, and
Theorem \ref{Teo-Potrie}.

\section{Numbering} 
The next section extensively references results in section 4 of \cite{HNil}.
To avoid any possible confusion between that paper and this one, we do not
state any results in this paper's section 4.

\section{Non-toral nilmanifolds}\label{SectionNilm}
There are several ways to view nilmanifolds in dimension three.
One is as quotients of nilpotent Lie groups.
Another is as bundles of the 2-torus over the circle.
While in this section, we only consider the former, we will use the latter
view in Appendix
\ref{Section-ClassificationofFoliations}.

Let $\cH$ denote the Heisenberg group, the group of all real-valued matrices
of the form
\[ \begin{pmatrix}
    1 & x & z \\ 0 & 1 & y \\ 0 & 0 & 1
    \end{pmatrix}.\]
Fix a co-compact subgroup $\Gamma$.  Then $\cH/\Gamma$ is a compact manifold
and every non-toral three-dimensional nilmanifold is of this form.  For
a more detailed introduction, see \cite{HNil}.
Observe that if $\pi:\Heis \to \bbR$ is a Lie group homomorphism, it must be of
the form
\[
    \begin{pmatrix}
    1 & x & z \\ 0 & 1 & y \\ 0 & 0 & 1
    \end{pmatrix}
    \mapsto a x + b y
\]
for constants $a,b \in \bbR$.  If $\pi$ is non-zero, its level sets define a
codimension one foliation $\tilde \FF_\pi$ on $\Heis$ which quotients down to a
foliation $\FF_\pi$ on the nilmanifold $\Heis / \Gamma$.  Plante showed that any
Reebless $C^2$ foliation of $\Heis / \Gamma$ must be almost aligned to $\FF_\pi$ for
some $\pi$ \cite{Plante}.
We need a similar result for foliations which are $C^{1,0+}$, that is, $C^0$
with $C^1$ leaves tangent to a $C^0$ distribution.
We prove the following.

\begin{prop} \label{nilfolns}
    Every Reebless $C^{1,0+}$ foliation on $\Heis / \Gamma$ is almost aligned
    with some $\FF_\pi$.
\end{prop}

There are two approaches to proving this.  One is to adapt the proof of Plante
to the $C^{1,0+}$ setting.  The other is to extend the techniques used by
Brin, Burago, and Ivanov on the 3-torus to other manifolds.  Both
approaches work, and since both techniques may be useful in the future if
applied to other manifolds, we give both proofs in the appendices.
Appendix \ref{sectionLattices} gives a self contained proof of this result using some algebraic arguments \`a la Brin-Burago-Ivanov.  Appendix \ref{Section-ClassificationofFoliations} provides a more geometric proof which uses the classification of such foliations in $\TT^3$ and some general position results (\cite{Roussarie,Gabai}) and gives a classification for every torus bundle over the circle (this will be used in \cite{HP}). 
In this section, we assume Proposition \ref{nilfolns} and show how it can be
used to prove Theorems \ref{nildyn} and \ref{nilconj}.

\begin{remark}
    To keep the presentation short, we only reprove those parts of
    \cite{HNil} which rely on absolute partial hyperbolicity.  This has the
    unfortunate consequence that the proof of dynamical coherence for the
    pointwise case is split between different papers and is not presented
    in one place from start to finish.
    As such, we present here a rough outline of the proof as a whole.

    If $J$ is a small unstable curve on the universal cover $\Heis$, then the
    length of $\tilde f^{k}(J)$ grows exponentially fast.  By the results
    of Brin, Burago, and Ivanov, $U_1(\tilde f^{k}(J))$ also
    grows exponentially fast in volume.  Here, $U_1(X)$ is all points at
    distance at most 1 from a point in $X$.  Comparing $\tilde f$ with its
    algebraic part $\Phi=\Phi_f:\Heis \to \Heis$,
    one can prove that such exponential expansion is only possible if $\tilde
    f^{k}(J)$ lies more-or-less in the unstable direction of $\Phi$.
    Since these unstable curves lie in leaves of the
    branching foliation $\tFFcub$, it follows from Proposition
    \ref{nilfolns} that $\tFFcub$ is almost aligned with the center-unstable
    foliation of $\Phi$ as no other foliations of the form $\tFF_\pi$
    have long curves in the same direction.

    Similarly, $\tFFcsb$ is almost aligned with the center-stable foliation
    of $\Phi$ and
    there is $R > 0$ such that
    if $p$ and $q$ lie on the same
    leaf of $\tFFcsb$ then $q$ is at distance at most $R$ from
    ${\tilde W}^{cs}_{\Phi}(p)$.
    If $\FFcsb$ is genuinely branching and not a true foliation,
    there are two distinct leaves of $\tFFcsb$ which
    intersect in a common point $p$ and one can find a small unstable arc $J$
    which lies in the space between these two leaves.
    For all $k>0$, $\tilde f^k(J)$ lies between two leaves of $\tFFcsb$ which pass
    through the point $\tilde f^k(p)$ and so all points in $f^k(J)$ are at
    distance at most $R$ from ${\tilde W}^{cs}_{\Phi}({f^k(p)})$.
    This contradicts the above-stated fact that $f^k(J)$ must grow
    exponentially fast in the unstable direction of $\Phi$.
    Thus, $\FFcsb$ is a non-branching foliation, and $f$ is dynamically
    coherent.  Again, we stress that this is a non-rigorous outline, and the
    rigorous proof follows below.
\end{remark}
\subsection{The orientable case}

Suppose $f: \Heis/\Gamma \to \Heis/\Gamma$ is partially hyperbolic.
To keep things simple, we make the unjustified assumption that the subbundles
$\Eu$, $\Ec$, $\Es$ are oriented, and that $Df$ preserves these orientations.
The next subsection shows how the general case can be deduced from this
special case.
Under these assumptions, Theorem \ref{branfol} applies and there are
$f$-invariant branching foliations $\FFcub$ and $\FFcsb$ tangent to $\Ecu$ and
$\Ecs$.

\begin{lemma} \label{nilalign}
    $\FFcub$ is almost aligned with some $\FF_\pi$.
\end{lemma}
\begin{proof}
    This follows immediately from Proposition \ref{nilfolns} and Theorem
    \ref{nearfol}.
\end{proof}
If $\pi:\Heis \to \bbR$ is a Lie group homomorphism, there is $C>0$ such that
\[
    | \pi(p) - \pi(q) | < C d(p,q)
\]
for all $p,q \in \Heis$.  This can be shown by adapting the proof of (4.2) in
\cite{HNil}.  With this, Lemma \ref{nilalign} can be restated thus.

\begin{lemma} \label{nilbounds}
    There is a Lie group homomorphism $\pi:\Heis \to \bbR$ and $R > 0$ such that
    $|\pi(p)-\pi(q)|<R$ for any two
    points $p$ and $q$ on the same leaf of $\tFFcub$.
\end{lemma}
Almost all of the work in \cite{HNil} carries over immediately to the
pointwise partially hyperbolic case.
In particular, the results in Section 4 of that paper up to and
including (4.8) also hold in our current setting.
That section defines two projections $\pi^s, \pi^u: \Heis \to \bbR$ which are Lie
group homomorphisms, and from their special properties, we can show the
following.


\begin{lemma} \label{cubdd}
    There is $R>0$ such that $|\pi^s(p) - \pi^s(q)| < R$ for all $p$ and $q$ on
    the same leaf of $\tFFcub$.
\end{lemma}
\begin{proof}
    By (4.3) of \cite{HNil}, there is a constant $x_0 > 0$ such that
    $|\pi^s(p)| < x_0$ implies $|\pi^s \tilde f(p)| < x_0$ for all $p \in \Heis$.
    As a consequence, we can find a small unstable segment $J$ such that
    $|\pi^s(p)| < x_0$ for all $n>0$ and $p \in \tilde f^n(J)$.
    By (4.7), the diameter of $\pi^u \tilde f^n(J)$ as a subset of $\bbR$ tends to
    infinity as $n \to \infty$.
    Each of the curves
    $\tilde f^n(J)$ is contained in a leaf of the lifted foliation $\tFFcub$.
    Further, it is not hard to show that the homomorphism $\pi:\Heis \to \bbR$
    given by Lemma \ref{nilbounds} above can be written as a linear
    combination, $\pi = a \pi^s + b \pi^u$.  The only way all of these estimates
    can hold is if $b=0$.  The result then follows.
\end{proof}


By repeating the arguments of Lemma \ref{cubdd} for $f \inv$ in place of $f$, the following
also holds.

\begin{lemma} \label{csbdd}
    There is $R>0$ such that $|\pi^u(p) - \pi^u(q)| < R$ for all $p$ and $q$ on the
    same leaf of $\tFFcsb$.
\end{lemma}
This is a non-dynamically coherent analogue of (4.10) of \cite{HNil}.
Using this version, we may repeat the proofs of (4.11), (4.12), and (4.13) of
\cite{HNil}, taking $\Wcs(p)$ to mean
\emph{any} leaf of $\tFFcsb$ which passes through the point $p$.  We state the
reformulations of (4.12) and (4.13) explicitly.

\begin{lemma} \label{halfgps}
    Every leaf of $\tFFcsb$ intersects every leaf of ${\tilde \cW}^u$ exactly
    once.
\end{lemma}
\begin{lemma} \label{endsunbdd}
    For $M>0$ there is $\ell>0$ such that for any unstable curve $J$ of
    length greater than $\ell$, the endpoints $p$ and $q$ satisfy
    $
        |\pi^u(p) - \pi^u(q)| > M.
    $  \end{lemma}
\begin{proof}
    [Proof of dynamical coherence.]
    We now prove dynamical coherence by showing that the ``branching''
    foliation $\tFFcsb$ does not actually branch.
    Suppose instead that two distinct leaves $L_1$ and $L_2$ pass through a
    point $p \in \Heis$ and take $q_1 \in L_1 \setminus L_2$.  By Lemma
    \ref{halfgps}, there is $q_2 \in L_2 \cap {\tilde W}^u(q_1)$.
    By Lemma \ref{endsunbdd}, $|\pi^u \tilde f^n(q_1) - \pi^u \tilde f^n(q_2)| \to \infty$
    as $n \to \infty$.
    By Lemma \ref{csbdd}, $|\pi^u \tilde f^n(q_1) - \pi^u \tilde f^n(q_2)| < 2R$ for all $n$, a
    contradiction.
    
    This same argument allows us to show that there is a unique $f$-invariant
    foliation.
    If $\cW^{cs}_1$ and $\cW^{cs}_2$ are distinct $f$-invariant foliations,
    there are distinct leaves $L_1$ and $L_2$ from the lifted foliations
    which pass through a common point $p \in \Heis$.
    Applying Lemmas \ref{endsunbdd} and \ref{csbdd} as above again gives a
    contradiction.
\end{proof}
\begin{proof}
    [Proof of leaf conjugacy.]
    As noted in \cite{HNil} (see the Remark after Lemma 4.10), only the proofs of (4.9) and (4.10) rely on
    absolute partial hyperbolicity.  The first of these is dynamical
    coherence, proved above, and once dynamical coherence is established,
    (4.10) is equivalent to Lemma \ref{csbdd}.  As such, the entire proof of
    leaf conjugacy now holds in the pointwise case.
\end{proof}

\subsection{The general case} 

The last subsection assumed the subbundles $\Eu$, $\Ec$ and $\Es$ had
orientations which were preserved by the derivative $Df$ of $f$.
Now consider the general case,
where the assumption does not necessarily hold.
For convenience, we simply say that a diffeomorphism preserves the orientation
of a bundle if its derivative preserves the orientation.

\begin{prop} \label{niloriented}
    If $f:\Heis/\Gamma \to \Heis/\Gamma$ is partially hyperbolic, then $\Eu$, $\Ec$,
    and $\Es$ are orientable.
\end{prop}
\begin{proof}
    Lift the subbundles to the universal cover $\Heis$ and choose an
    orientation for each of them.         
    Regarding $\Gamma$ as the group of deck transformations,
    those $\gamma \in \Gamma$ which preserve all three of these orientations
    form a normal, finite-index subgroup $\Gamma_0$.
    Lift $f$ to $\tilde f:\Heis \to \Heis$.
    As $\tilde f^2$ preserves the orientations of the subbundles, it descends
    to a diffeomorphism $g$ on the nilmanifold $\Heis/\Gamma_0$.
    On this manifold, the subbundles are oriented and $g$ preserves these
    orientations.  Therefore, all of the previous analysis now applies.
    In particular, Lemma \ref{endsunbdd} holds for the unstable foliation
    on $\Heis$.

    Choose an orientation for $\bbR$.
    Suppose $x \in \Heis$ and that $\alpha:\bbR \to \tilde \cW^u(x)$ is an
    orientation-preserving
    parameterization of the unstable curve through $x$.
    Lemma \ref{endsunbdd} implies that
    $\lim_{t \to +\infty} \pi^u \alpha(t)$
    is either $+\infty$ or $-\infty$.
    Basic continuity arguments show that this limit does not depend on
    the choice of $\alpha$ or $x$.
    It is straightforward to show that
    \[
        \lim_{t \to +\infty} \pi^u \gamma \alpha(t) = \lim_{t \to +\infty} \pi^u \alpha(t)
    \]
    for any deck transformation $\gamma \in \Gamma$.
    From this, it follows that the orientation of the lifted unstable bundle is
    $\Gamma$-invariant and descends to an orientation of $\Eu$ on $\Heis/\Gamma$.
    A similar argument shows that $\Es$ is orientable.  As the full
    three-dimensional tangent bundle $\Eu \oplus \Ec \oplus \Es$ is orientable, the
    center bundle $\Ec$ is orientable as well.
\end{proof}
Though the subbundles are always orientable, there are examples of partially
hyperbolic systems on $\Heis/\Gamma$ which do not preserve these orientations.
(To construct such an example, take a hyperbolic toral automorphism on $\bbT^2$
which does not preserve the orientations of its stable and unstable
subbundles.  Then, build a bundle map over this.)
Fortunately, this obstruction does not pose a serious problem.

Suppose $f:\Heis/\Gamma \to \Heis/\Gamma$ is partially hyperbolic.
Then, $f^2$ preserves the orientations of the bundles $\Eu$, $\Ec$, and $\Es$,
and from the above proof of dynamical coherence, there is a unique
$f^2$-invariant foliation $\cW$ tangent to $\Ecs$.
Note that $f(\cW)$ is also an $f^2$-invariant foliation tangent to
$\Ecs$ and so $f(\cW)=\cW$.
Finding such an $f$-invariant foliation via Theorem \ref{branfol} was the only
reason for the added assumption on orientations, and therefore the results of
the last subsection now follow in the general (non-orientation-preserving)
case.

\subsection{Finite quotients} 

To prove Theorem \ref{alldyn} from Theorem \ref{nildyn} we must consider
finite quotients of nilmanifolds.

\begin{prop}
    If a compact manifold $M$ has a universal cover homeomorphic to $\bbR^3$
    and $\pi_1(M)$ is virtually nilpotent, then
    there is a regular finite cover $\hat M$ over $M$ such that
    $\hat M$ is diffeomorphic to a nilmanifold and every diffeomorphism
    $f:M \to M$ can be lifted to a diffeomorphism $\hat f:\hat M \to \hat M$.
\end{prop}
\begin{proof}
    As $\pi_1(M)$ is virtually nilpotent, the Hirsch-Plotkin radical $N$
    is the maximal normal nilpotent subgroup and is of finite index.
    It defines a regular finite covering $\hat M$.
    As any automorphism of $\pi_1(M)$ leaves $N$ invariant, any diffeomorphism
    of $M$ lifts to $\hat M$.
    By a classical result of P.~A.~Smith, any free action on $\bbR^n$ is
    torsion free (see \cite[page 43]{borel-transformation}), and then by a
    result of Malcev \cite{malcev}, $N$ can be identified with the fundamental
    group of a compact nilmanifold $\hat K$.
    As $\hat M$ and $\hat K$ are aspherical and have isomorphic fundamental
    groups, they are homotopy equivalent and $\hat K$ must be
    three-dimensional.
    One can verify
    directly that three-dimensional nilmanifolds are ``sufficiently large'' in
    the sense of Waldhausen, and therefore, $\hat M$ and $\hat K$ are in fact
    homeomorphic \cite{waldhausen}.
    By a result of Moise, they are diffeomorphic \cite{moise}.
\end{proof}

\begin{proof}[Proof of Theorem \ref{alldyn}]
If $f:M \to M$ is as in Theorem \ref{alldyn}, it lifts to
$\hat f:\hat M \to \hat M$ by the above proposition, and by Theorem \ref{nildyn}
there is a foliation
$\cW$ which for each $k \ge 1$ is the unique
$\hat f^k$-invariant foliation tangent to $\Ecs$.
For each deck transformation $\alpha$ of the finite covering of $\hat M$
over $M$, there is $k$ such that $\hat f^k \alpha = \alpha \hat f^k$, so that
$\alpha(\cW)$ is also $\hat f^k$-invariant, and by uniqueness
$\alpha(\cW) = \cW$.
As such, $\cW$ quotients to a foliation on $M$ and is the unique
$f$-invariant foliation tangent to $\Ecs$.
\end{proof}


\section{Dynamical consequences}\label{Section-Consequences} 

\subsection{Entropy maximizing measures} 
In this subsection, we observe how our results allow one to re-obtain some
results about entropy maximizing measures previously proven under less
generality.
First, the main result of \cite{Ures} also holds in the pointwise case.

\begin{prop}\label{Prop-EMM} Let $f: \TT^3 \to \TT^3$ be a partially hyperbolic diffeomorphism such that its linear part $A_f$ is hyperbolic. Then, $f$ has a unique maximal entropy measure which is measurably equivalent to the volume measure for $A_f$ (i.e. it is \emph{intrinsically ergodic}). 
\end{prop}

\begin{proof}
The only dependence on absolute partial hyperbolicity in \cite{Ures} is in
establishing Lemmas 3.2 and 3.3.
Those lemmas can be proven directly from Corollary \ref{isoAnosov} of this
paper.
\end{proof}

In higher dimensions, under the assumption of being isotopic to Anosov along a path of partially hyperbolic systems it is possible to recover also the same result in the pointwise case (\cite{FPS}). 

\begin{prop}
Let $f$ be a $C^{1+\alpha}$ partially hyperbolic diffeomorphism of a
non-toral three-dimen\-sional nilmanifold.
Then, $f$ has at least one and at most a finite number of measures of maximal
entropy.
\end{prop}

\begin{proof}
By Theorem \ref{nildyn}, $f$ has a center foliation of compact leaves,
and is also accessible (see Proposition \ref{nilacc} below).
The proposition then follows immediately from the main result of \cite{HHTU}.
See that paper for more details.
\end{proof}

Under additional assumptions, one can prove similar results for diffeomorphisms
on $\TT^3$ with non-hyperbolic linear part.  We again refer the interested
reader to \cite{HHTU}.

\subsection{Uniqueness of attractors on nilmanifolds} 

In this section we deduce some dynamical consequences from the
main results to the study of dynamics of partially hyperbolic
diffeomorphisms in 3 dimensional non-toral nilmanifolds.

We recall that a \emph{quasi-attractor} is a chain-recurrence
class\footnote{See \cite{BDV} Chapter 10.} such that it admits a
basis of neighborhoods $\{U_n\}$ such that $f(\overline U_n) \en
U_n$. It is well known (and simple to show) that a quasi-attractor
must be saturated by unstable sets. Thus, in the partially
hyperbolic setting saturated by the unstable foliation.

As a consequence of an argument very similar to the one in
\cite{BCLJ} (Proposition 4.2) we can deduce that for a partially
hyperbolic diffeomorphism of $\cH/\Gamma$ these foliation have a unique
minimal set, thus concluding:

\begin{prop} Let $f: \cH/\Gamma \to \cH/\Gamma$ be a strong partially hyperbolic
diffeomorphism, then, $f$ has a unique quasi-attractor. Moreover,
there exists a leaf of $\cW^u$ which is dense in $\cH/\Gamma$ and a unique minimal set of the foliation.
\end{prop}

One could wonder if such a result does not imply that every strong partially hyperbolic diffeomorphism on $\Heis/\Gamma$ is transitive, however, this result is sharp: Y. Shi has constructed examples of partially hyperbolic diffeomorphisms in nilmanifold which are Axiom A (in particular they have an attractor and a repeller). The construction is in the spirit of \cite{BG}.

\proof The key point is that this problem can be reduced to a
problem about dynamics of homeomorphisms of $\TT^2$ which are
skew-products over the irrational rotation. In fact, we can
consider a two-dimensional torus $T$ consisting of center leaves
and transverse to both the strong stable and strong unstable
foliations of $\Phi_f$, the algebraic part of $f$.

We consider the leaf conjugacy $h: \cH/\Gamma \to \cH/\Gamma$ from $f$ to $\Phi_f$. 

We obtain that $h^{-1}(T)$ is a topological torus, foliated by
center leaves and such that it is transverse to both strong
foliations for $f$. Now, the return map of the unstable holonomy
defines a homeomorphism of $T$ which after conjugacy can be
written in the following form:

$$ F(x,y) = (x + \alpha, \varphi_x(y)) \ \mod \ZZ^2  $$

\noindent and satisfying that $\varphi_x(y+1)=\varphi_x(y) + 1$
and that $\varphi_{x+1}(y)= \varphi_{x}(y)+k$ (this is to say that
$F$ is homotopic to a Dehn twist).

We will divide the proof in two claims about such homeomorphisms.
Notice that the existence of a unique minimal set for the unstable
foliation implies immediately that $f$ has a unique
quasi-attractor since quasi-attractors are compact disjoint and
saturated by the unstable foliation.

The following claim\footnote{We thank Tobias J\"ager who
communicated to us this argument.} implies that $f$ has a dense
unstable leaf:

\begin{af} $F$ is transitive.
\end{af}

{ 
\renewcommand{\qedsymbol}{$\diamondsuit$}
\begin{proof}
Consider two open sets $U_1, U_2$ in $\TT^2$. We will show
that there is an iterate of $U$ that intersects $V$. To do this,
we consider arcs $J_i \en U_i$ of the form $J_i= (x_i^-,x_i^+)
\times \{y_i\}$.

We will work in the universal cover $\RR^2$ of $\TT^2$. We claim
that that if we denote as $\tilde F$ to the lift of $F$ and we
consider the iterates $\tilde J_1^k=\tilde F^k(\tilde J_k)$ and we
define $\ell_k$ to be the length of $\tilde J_1^k$ when projected
in the second coordinate, then $\ell_k$ goes to infinity with $k$.
Now, since the dynamics on the base is minimal this implies that
there is an iterate of $F$ such that $F(J_1)\cap J_2 \neq
\emptyset$ which will conclude the proof.

To prove the claim, consider a sequence of iterates $J_1,
f^{i_1}(J_1), \ldots f^{i_m}$ such that the projection of their
union in the first coordinate is surjective. It is clear that this
can be done since the dynamics in the first coordinate is minimal.
Now, join these arcs by arcs contained in the second coordinate in
order to make a non-trivial loop $\gamma$ homotopic to the first
coordinate circle. The iterates of $\gamma$ by $F$ start to turn
around the second coordinate infinitely many times and thus the
length of $F^n(\gamma)$ goes to infinity. since the vertical arcs
remain of bounded length, we conclude the claim.
\end{proof}
} 

The following claim concludes the proof of the Proposition.

\begin{af} $F$ has a unique minimal set.
\end{af}

{ 
\renewcommand{\qedsymbol}{$\diamondsuit\,\Box$}
\begin{proof}
We follow the argument in Proposition 4.2 of \cite{BCLJ}
which with the use of the previous proposition allows to conclude.

Assume there are two disjoint minimal sets $K, K'$ in $\TT^2$. We
know that $K \cap K' = \emptyset$ and since they are compact, they
are at bounded distance from below. Consider the set $U_1$ as the
set of (ordered) arcs $\{x\} \times (y_1,y_2)$ with $y_1 \in K$
and $y_2 \in K'$. Similarly one can consider the set of arcs $U_2$
of the form $\{x\} \times (y_1,y_2)$ with $y_1 \in K' $ and $y_2
\in K$. Clearly, $U_1 \cap U_2 = \emptyset$, and since the maps
$\varphi_x$ are order preserving they are invariant. Both sets
intersect every fiber of the type $\{x\} \times S^1$ since the
base dynamics is minimal.

We will prove that both have non-empty interior which will
contradict transitivity. To do this, notice that compactness of
$K$ and $K'$ implies that if we consider the mappings $x \mapsto
(K \cap \{x\} \times S^1)$ and $x \mapsto (K' \cap \{x\} \times
S^1)$ which are both semi-continuous and thus share a residual set
of continuity points. Take $x$ a common continuity point, then any
point of the form $(x,y)$ in $U_i$ is an interior point of $U_i$
concluding the proof.
\end{proof}
} 

\subsection{Accessibility} 

A partially hyperbolic diffeomorphism $f:M \to M$ is \emph{accessible} if
any two points $x,y \in M$ can be
connected by a concatenation of paths, each path tangent either to $\Eu$ or
$\Es$.
As with dynamical coherence, we can now show there are manifolds where every
pointwise partially hyperbolic system is accessible.

\begin{prop} \label{nilacc}
    Suppose $M$ is a non-toral 3-dimensional nilmanifold.
    Every $C^1$ partially hyperbolic diffeomorphism on $M$ (measure-preserving
    or not) is accessible.
\end{prop}
This was previously proven in the measure-preserving case in
\cite{HHUPHdim3}.

\begin{proof}
    If $f$ is not accessible, there is a non-empty lamination $\Lambda \subset
    M$ whose leaves are complete and tangent to $\Eu \oplus \Es$
    \cite[\S3]{RHRHU-accessibility}.
    Lift $\Lambda$ to the universal cover $\Heis$ and choose a leaf $S$ of the
    lifted lamination.
    By Global Product Structure, as proved in \cite{HNil} and now extended
    to the pointwise case, such a surface $S$ intersects each center leaf
    exactly once.
    For any deck transformation $\gamma \in \Gamma$, the surface $\gamma(S)$
    also intersects each center leaf exactly once.
    Further $\gamma(S)$ and $S$ are either disjoint or coincide.
    By Proposition \ref{niloriented}, $\Ec$ is orientable on $\Heis/\Gamma$,
    and so there is a $\Gamma$-invariant orientation of $\Ec$ on $\Heis$.
    This orientation allows us to define an ordering on the surfaces
    $\gamma(S)$ for $\gamma \in \Gamma$, saying whether one such surface is
    ``above'' or ``below'' another.
    From Proposition \ref{nilorder} in Appendix \ref{sectionLattices}, this
    ordering must be trivial.
    However, as shown in \cite{HNil} there are elements $\gamma \in \Gamma$
    which fix each center leaf, but do not fix any point of $\Heis$.
    This gives a contradiction.
\end{proof}

\appendix
\section{Classification of Reebless foliations in nilmanifolds}\label{sectionLattices} 

In this section, a \emph{leaf system} is a tuple $(M, \hat M, \FF, S)$
where
\begin{itemize}
    \item $M$ is a compact manifold without boundary,
    \item $\hat M$ is a normal (or regular) covering space of $M$,
    \item $\FF$ is a transversely oriented foliation on $M$,
    \item $S$ is a leaf of $\hat \FF$, the foliation obtained by lifting $\FF$ to
    $\hat M$, and
    \item there is no closed loop topologically transverse to $\hat \FF$ passing
    through $S$.
\end{itemize}
While the last condition implies that $S$ is properly embedded in $\hat M$,
other leaves of $\hat \FF$ may not be.  Also note that $\hat M$ is not
necessarily the universal cover of $M$.
Assume such a leaf system is fixed, and let $L$ denote the group of deck
transformations on $\hat M$.  If $\gamma \in L$, let $x \gamma$ denote the
action of $\gamma$ on $x \in \hat M$.
A right action was chosen to keep the notation as close as possible to that
used in \cite{BBI2}.
$S$ splits $\hat M$ into two open subspaces,
$S_+$ and $S_-$ where the sign is given by the transverse orientation on $\FF$.
Define
\begin{align*}
    \Gamma_+ &= \{ \gamma \in L : S_+ \gamma \subset S_+\}, \\
    \Gamma_- &= \{ \gamma \in L : S_+ \gamma \supset S_+\}, \\
    \Gamma &= \Gamma_+ \cup \Gamma_-.
\end{align*}
\begin{lemma}
    [confer Lemma 3.9 of \cite{BBI2}]
    If $A$ is an abelian subgroup of $L$, then $A \cap \Gamma$ is a subgroup of
    $A$.
\end{lemma}
For $X \subset \hat M$ and $H \subset L$, define the notation
$
    X H := \{ x \gamma : x \in \hat M,\ \gamma \in H\}.
$
\begin{lemma}
    [confer Lemma 3.11 of \cite{BBI2}] \label{abelsub}
    If $A$ is an abelian subgroup of $L$
    and $S_+ A = S_- A = \hat M$
    then $A \subset \Gamma$.
\end{lemma}
\begin{lemma} \label{quotient}
    Suppose $(M, \tilde M, \FF, S)$ is a leaf system, where $\tilde M$ is the
    universal cover of $M$, $N$ is a normal subgroup of the group of deck
    transformations and
    $
        \tilde U := S_+ N  \ne  \tilde M.
    $
    Then, for $\hat M = \tilde M / N$ and $\hat U = \tilde U / N$, and any
    connected component $C$ of the (non-empty) boundary of $\hat U$, the tuple
    $(M, \hat M, \FF, C)$ is a leaf system.
          \end{lemma}
\begin{proof}
    As $\tilde U$ is $N$-invariant, it is easy to see that
    $\varnothing  \ne  \hat U  \ne  \hat M$ 
    and that $\hat U$ is open and is a union of leaves of the lift $\hat \FF$
    of the foliation $\FF$.  Further, while the boundary of $\hat U$ many have
    several components, each has the same orientation, and so any curve
    transverse to $\hat \FF$ which exits $\hat U$, cannot later re-enter $\hat
    U$.  This shows the last item in the definition of ``leaf system'' and the
    others are immediate.
\end{proof}
\begin{lemma}
    [confer Lemma 3.12 of \cite{BBI2}] \label{lzd}
    Suppose the leaf system is such that there is a group isomorphism
    $h:L \to \Zd$ and $S_+ L = S_- L = \hat M.$
    Then, $\Gamma = L$ and there is a hyperplane $P$ separating $\Rd$ into
    closed half-spaces $H_+$ and $H_-$ such that $h(\Gamma_+) \subset H_+$ and
    $h(\Gamma_-) \subset H_-$.
\end{lemma}
All of the results listed so far hold for any $C^0$ foliation $\FF$.  To
proceed further, we need to add an assumption of ``uniform structure.''

\begin{assumption}
    There are constants $r,R>0$ such that if $x \in S_+$ then there is $y \in
    \hat M$ such that
    $
        B_r(y) \subset B_R(x) \cap S_+.
    $
    Similarly, if $x \in S_-$, there is $y \in \hat M$ such that
    $
        B_r(y) \subset B_R(x) \cap S_-.
    $  \end{assumption}
This property holds for any foliation tangent to a $C^0$ subbundle
of the tangent bundle, so long as the base space $M$ is compact.

\begin{lemma} \label{nets}
    If $X = \{x_1, \ldots, x_n\}$ is an $r$-net of a fundamental domain of
    $\hat M$, then
    $
        X L \cap S_+
    $
    is an $R$-net for $S_+$.
    Further, if $\Gamma = L$ and $X \cap S_+ = \varnothing$, then
    $
        X \Gamma_+
    $
    is an $R$-net for $S_+$.
\end{lemma}
\begin{proof}
    The first part follows immediately from the above assumption and the fact
    that $X L$ is an $r$-net for all of $\hat M$.
    To prove the second part, suppose $y \in S_+$.
    By the first part, there is $x_i \in X$ and $\gamma \in L$
    such that $y \in B_R(x_i \gamma)$ and $x_i \gamma \in S_+$.
    Then, $x_i$ demonstrates that $S_+ \gamma \inv$ is not a subset of $S_+$,
    which, in the special case of
    $L = \Gamma = \Gamma_+ \cup \Gamma_-$
    implies that $\gamma \in \Gamma_+$.
      \end{proof}
We say that a codimension one manifold $P$ has the same ``ordering'' as
$S$ if $P$ splits $\hat M$ into two subspaces $P_+$ and $P_-$ where
$P_+ \gamma \subset P_+$
for all $\gamma \in \Gamma_+$ and 
$P_+ \gamma \supset P_+$
for all $\gamma \in \Gamma_-$.

\begin{prop} \label{sameorder}
    Suppose $\Gamma = L$ and $S_+ \Gamma = S_- \Gamma = \hat M$ and that there
    is a codimension one manifold $P$ with the same ``ordering'' as $S$.  Then
    $S$ lies a finite distance from $P$.
\end{prop}
\begin{proof}
    Since $S_- \Gamma = \hat M$, we can find a fundamental domain for $\hat
    M$ inside of $S_-$.  Let $X = \{x_1, \ldots, x_n\}$ be an $r$-net for this
    fundamental domain, where $r$ is as in Lemma \ref{nets}.
    Fix a point $y \in P_+$ and set $D = \sup_i d(x_i, y)$ where $d$ is distance
    measured on $\hat M$.  Then, for $\gamma \in \Gamma_+$,
    $d(x_i \gamma, y \gamma)  \le  D$, and $x_i \gamma \in P_+ \gamma \subset P_+$.
    This means that
    every point of $X \Gamma_+$, is at distance at most $D$ from $P_+$.  By
    Lemma \ref{nets}, $X \Gamma_+$ is an $R$-net for $S_+$ and so every point in
    $S_+$ is at distance at most $R+D$ from $P_+$.
    A similar argument shows that every point in $S_-$ is a bounded distance
    from $P_-$ and completes the proof.
\end{proof}
Now consider the specific case of the Heisenberg group.

\begin{prop} \label{nilorder}
    If a relation ``$ \le $'' on the Heisenberg lattice
    \[    
        \Gamma = \langle x,y,z : x y = y x z, x z=z x, y z=z y \rangle
    \]
    satisfies all of the following properties
    \begin{itemize}
        \item reflexivity: $u  \le  u$,
        \item totality: either $u  \le  v$ or $v  \le  u$ or both,
        \item transitivity: $u  \le  v$ and $v  \le  w$ implies $u  \le  w$,
        \item right-invariance: $u  \le  v$ implies $u w  \le  v w$, and
        \item an ``Archimedean'' property for $z$:
        for all $u \in \Gamma$, there is $k(u) \in \bbZ$ such that
        \[
            z ^ {k(u)}  \le  u  \le  z ^ {k(u)+1},
        \]  \end{itemize}
    then $ \le $ is trivial: $u  \le  v  \le  u$ for all $u$ and $v$.
\end{prop}
\begin{proof}
    Using the Archimedean property, let $i,j \in \bbZ$ be such that
    $z ^ i  \le  x  \le  z ^ {i+1}$
    and
    $z ^ j  \le  y  \le  z ^ {j+1}$.
    Since $z$ commutes with both $x$ and $y$, for any $n \in \bbZ$
    \[
        z^{(i+j)n}  \le  x^n y^n  \le  z^{(i+j+2)n}
    \]
    and
    \begin{align*}
        z^{(i+j)n}  &\le  y^n x^n \phantom{z^{n^2}}  \le  z^{(i+j+2)n}  \quad \Rightarrow \\  
        z^{(i+j)n+n^2}  &\le  y^n x^n z^{n^2}  \le  z^{(i+j+2)n+n^2}.
    \end{align*}
    Notice that $x^n y^n = y^n x^n z^{n^2}$ and therefore
    \[
        z^{(i+j)n+n^2}  \le  z^{(i+j+2)n}  \quad \Rightarrow \quad  z^{n^2}  \le  z^{2n}
          \]
    for all $n \in \bbZ$.
    This is enough to deduce $z^{k_1}  \le  z^{k_2}$ for any two integers
    $k_1, k_2$, and then by the Archimedean property
    \[
        u  \le  z ^ {k(u) + 1}  \le  z ^ {k(v)}  \le  v
    \]
    for any $u$ and $v$, showing that the relation is trivial.
\end{proof}



%
Any three-dimensional nilmanifold can be thought of as a circle bundle over a
2-torus.
If the bundle is trivial, the manifold is the 3-torus.  If it is
non-trivial, then it can be finitely covered by a manifold $M$ whose
fundamental group is exactly as in Property \ref{nilorder}.  Therefore, we will
only consider this specific case $M$.
The base space $\bbT^2$ has universal cover $\bbR^2$.  Therefore, the bundle $M$
is covered by a circle bundle $\hat M$ whose base space is $\bbR^2$.
Both $M$ and $\hat M$ have as universal cover the Heisenberg group, here
denoted by $\tilde M$.
For a Lie group homomorphism $\pi: \tilde M \to \bbR$, recall the definition of
the foliation $\FF_\pi$ where each leaf of $\tFF_\pi$ is a level-set of
$\pi$.
%

\begin{thm}
    If $M$ is a non-toral three-dimensional nilmanifold, and $\FF$ is a $C^0$
    Reebless foliation on $M$ with uniform structure (as explained above),
    then $\FF$ is almost aligned with $\FF_\pi$ for some $\pi$.
\end{thm}
\begin{proof}
    Without loss of generality, we may replace $M$ by a finite cover,
    such that the fundamental group of $M$ is exactly the group $\Gamma$ given
    in Proposition \ref{nilorder}.
    As $\FF$ is Reebless, there are no closed cycles transverse to
    the lifted foliation $\tilde \FF$.
    Fix a lifted leaf $S$.
    This defines a leaf system $(M, \tilde M, \FF, S)$.
    As before, let $L$ denote the set of deck transformations on $\tilde M$.
    Let $Z \subset L$ be the center of this group.

    We first assume that $S_+ Z = S_- Z = \tilde M$ and show this leads to a
    contradiction.  Under this assumption, pick any $\gamma \in L$ and define
    $A_\gamma$ as the smallest subgroup of $L$ containing $Z$ and $\gamma$.  As
    this group is abelian and $S_+ A_\gamma = S_- A_\gamma = \tilde M$,
    Lemma \ref{abelsub} implies that $A_\gamma \subset \Gamma$.  As $\gamma$ was
    chosen arbitrarily, this shows that $\Gamma=L$.
    
    The center $Z$ is cyclic, and we may take a generator $z$ such that $S_+
    \subset S_+ z$.
    Consider $\gamma \in L$ and take points $x \in
    S_+ \gamma$ and $y \in \tilde M \setminus S_+ \gamma$.
    By the assumption $S_+ Z = S_- Z = \tilde M$,
    there are integers $i < j$ such that $x \in \tilde M \setminus S_+ z^i$
    and $y \in S_+ z^j$.  Using that $L = \Gamma = \Gamma_+ \cup \Gamma_-$,
    \[
        S_+ z^i \subset S_+ \gamma \subset S_+ z^j.
    \]
    Hence, there is $k \in \bbZ$, $i  \le  k < j$ such that
    \[
        S_+ z^k \subset S_+ \gamma \subset S_+ z^{k+1}.
    \]
    Define a relation ``$ \le $'' on $\Gamma = L$ by $\alpha  \le  \beta$ if
    $S_+ \alpha \subset S_+ \beta$.
    This relation satisfies the hypotheses of Proposition \ref{nilorder} and is
    therefore trivial: $S_+ \gamma = S_+$ for all $\gamma \in \Gamma$.  This
    contradicts the assumption $S_+ Z = \tilde M$.

    We have reduced to the case where either $S_+ Z  \ne  \tilde M$ or
    $S_- Z  \ne  \tilde M$.  Without loss of generality, assume the first.
    Consider now a new leaf system $(M, \hat M, \FF, C)$
    given by Lemma \ref{quotient}.  The group of deck transformations
    $\hat L = L / Z$ is isomorphic to $\bbZ^2$.

    First consider the case $C_+ \hat L  \ne  \hat M$.  The boundary of $C_+ \hat L$
    quotients to a union of compact leaves of $M$.  Let $T$ be such one
    such leaf.
    As $\FF$ is Reebless, the embedding $T \hookrightarrow M$ is
    $\pi_1$-injective, and this implies that $T$ is either
    a 2-sphere, or a 2-torus.
    The Reeb stability theorem rules out the
    possibility of a sphere.
    Then, as a $\pi_1$-injective torus, $T$ lifts to a cylinder
    $\hat T \subset \hat L$ and one can verify
    $\hat T_+ \hat L =  \hat T_- \hat L = \hat M$.
    Therefore, up to replacing the leaf $C$, we may freely assume that
    $C_+ \hat L =  C_- \hat L = \hat M$.

    Applying Lemma \ref{lzd} and Proposition \ref{sameorder}, the leaf $C$ is a
    finite distance away from a surface $\hat P$ which is the
    pre-image by the projection $\hat M \to \bbR^2$ of a geometric line on
    $\bbR^2$.

    Fix a fundamental domain $K$ of the covering $\hat M \to M$.
    Since $C_+ \hat L = C_- \hat L = \hat M$, $K$ must lie inside a set
    $U := C_+ \alpha \cap C_- \beta$ for some $\alpha, \beta \in
    \hat L$.  Further, there is $D > 0$ such that every point of $U$ is at
    distance at most $D$ from $\hat P$.
    Since there are no topological crossings, any leaf $C'$ of $\hat \FF$ passing
    though $K$ must lie in the closure of $U$ and so
    $C'$ lies at most a distance $D$ away from $\hat P$.
    By considering translates $C' \gamma$ for $\gamma \in \hat L$, we
    can show that every leaf of $\hat \FF$ lies at most a distance $D$ away from
    some translate $\hat P \gamma$ of $\hat P$.
    There is a Lie group homomorphism $\pi:\tilde M\to\bbR$ such that each
    translated surface $\hat P \gamma$ lifts to a leaf of $\tFF_\pi$.
    Then, lifting leaves of $\hat \FF$ to leaves of $\tilde \FF$, the result is
    proved.
\end{proof}

\section{Classification of foliations in torus bundles over the circle} 
\label{Section-ClassificationofFoliations}

In this section we give a classification result of foliations of 3-dimensional manifolds which are torus bundles over the circle $S^1$. This result applied to mapping torus of Anosov automorphisms will be also used in \cite{HP}. 

We will classify Reebless foliations of such manifolds under the relation of being almost aligned with some model foliation. In certain cases, namely, when there are no torus leaves, we will be able to obtain a stronger relation.  

Two branching foliations $\FF_1$ and $\FF_2$ are \emph{almost parallel} if there exists $R>0$ such that:

\begin{itemize}
\item[-] For every leaf $L_1 \in \tFF_1$ there exists a leaf $L_2 \in \tFF_2$ such that $L_1 \en B_R(L_2)$ and $L_2 \en B_R(L_1)$ (i.e. the Hausdorff distance between $L_1$ and $L_2$ is smaller than $R$).
\item[-] For every leaf $L_2 \in \tFF_2$ there exists a leaf $L_1 \in \tFF_1$ such that $L_1 \en B_R(L_2)$ and $L_2 \en B_R(L_1)$.
\end{itemize}

Foliations $\cF_1,\cF_2$ and $\cF_4$ in Figure 1 are almost parallel to each other, however, $\cF_3$ is not almost parallel to them. 

We state the definition for branching foliations since in fact, the results of \cite{BI} give that the branching foliations they construct are almost parallel to some Reebless foliation. Moreover, we have the following:

\begin{prop}\label{Prop-PropertiesAlmostParallel} The following properties are verified:
\begin{itemize}
\item[(i)] Being almost parallel is an equivalence relation among branching foliations. 
\item[(ii)] If $\cW$ is almost aligned with $\cW'$
a foliation in $M$ and $\varphi$ a diffeomorphism of $M$ isotopic
to the identity, then $\varphi(\cW)$ is almost parallel to $\cW$ and almost aligned to $\cW'$.
\end{itemize}
\end{prop}

\dem Property (i) follows from the triangle inequality. Properties
(ii) follows from the fact that the map in the universal
cover is at bounded distance from the identity.
\lqqd

For $C^2$-foliations, Plante (see \cite{Plante}) gave a
classification of foliations without torus leaves in
$3$-dimensional manifolds with almost solvable fundamental group.
His proof relies on the application of a result from
\cite{Roussarie} which uses the $C^2$-hypothesis in an important
way (other results which used the $C^2$-hypothesis such as
Novikov's Theorem are now well known to work for $C^0$-foliations
thanks to \cite{Solodov}). We shall use a recent result of Gabai
\cite{Gabai} which plays the role of Roussarie's result and allows
the argument  of Plante to be recovered.

We state now a consequence of Theorem 2.7 of \cite{Gabai} which
will serve our purposes\footnote{Notice that a foliation of a
$3$-dimensional manifold without closed leaves is \emph{taut}, see
\cite{Calegari} Chapter 4 for definitions and these results.}:

\begin{thm}\label{Teo-Gabai} Let $\cF$ be a foliation of a 3-dimensional manifold
$M$ without closed leaves and let $T$ be an embedded
two-dimensional torus whose fundamental group injects in the one
of $M$, then, $T$ is isotopic to a torus which is transverse to
$\cF$.
\end{thm}

On the one hand, Gabai proves that a closed incompressible surface
is homotopic to a surface which is either a leaf of $\cF$ or
intersects $\cF$ only in isolated saddle tangencies. Since the
torus has zero Euler characteristic, this implies that it must be
transverse to $\cF$. We remark that Gabai's result is
stated by the existence of a homotopy, and this must be so since
Gabai starts with an \emph{immersed} surface, however, it can be
seen that Lemma 2.6 of \cite{Gabai} can be done by isotopies if
the initial surface is embedded. The rest of the proof uses only
isotopies. See also \cite{Calegari} Lemma 5.11 and the Remark
after Corollary 5.13.

In view of this result and in order to classify foliations in torus bundles over the circle it is natural to look at foliations of $\TT^2 \times [0,1]$. By considering a gluing of the boundaries by the identity map, we get a foliation of $\TT^3$. These foliations (in the $C^0$-case) were classified in \cite{Pot}, and the result can be restated in the terms used in this paper as follows: 

\begin{thm}[Theorem 5.4 and Proposition 5.7 of \cite{Pot}]\label{FoliationT3} Let $\cW$ be a Reebless foliation of $\TT^3$, then, $\cW$ is almost aligned with a linear foliation of $\TT^3$. Moreover, if the linear foliation is not a foliation by tori, then $\cW$ is almost parallel to the linear foliation and if it is a foliation by tori then there is at least one torus leaf.
\end{thm}

A \emph{linear foliation} of $\TT^3$ is the projection by the natural projection $p: \RR^3 \to \RR^3/\ZZ^3 \cong \TT^3$ of a linear foliation of $\RR^3$. It is a foliation by tori if the linear foliation is given by a plane generated by two vectors in $\ZZ^3$. The same classification can be done for one-dimensional foliations of $\TT^2$ for which the proof is easier (see for example section 4.A of \cite{Pot2}). This allows us to classify foliations of $\TT^2 \times [0,1]$ transverse to the boundary: 

\begin{prop}\label{Prop-FolT2int} Let $\cW$ be a foliation of $\TT^2 \times [0,1]$ which is transverse to the boundary and has no torus leaves. Then, the foliation $\cW$ is almost aligned with a foliation of the form $\cL \times [0,1]$ where $\cL$ is a linear foliation of $\TT^2$. If $\cL$ is not a foliation by circles, then $\cW$ is almost parallel to $\cL \times [0,1]$. 
\end{prop}

\begin{proof} The proof can be done directly (see \cite{Roussarie,Plante} for the $C^2$-case). We will use Theorem \ref{FoliationT3} instead. Consider the foliation $\cW'$ of $\TT^2 \times [0,2]$ obtained by gluing $\TT^2 \times [0,1]$ with the foliation $\cW$ with $\TT^2 \times [1,2]$ with the foliation $\varphi(\cW)$ where $\varphi_1: \TT^2 \times [0,1] \to \TT^2 \times [1,2]$ is given by $\varphi_1(x,t)=(x,2-t)$. It is not hard to check that this gives rise to a well defined foliation $\cW'$ of $\TT^2 \times [0,2]$ (is like putting a mirror in the torus $\TT^2 \times \{1\}$). 

We can now construct a foliation of $\TT^3$ as follows: we glue $\TT^2 \times \{0\}$ with $\TT^2 \times \{2\}$ by the diffeomorphism $\varphi_2: \TT^2 \times \{0\} \to \TT^2 \times \{2\}$ given by $\varphi_2(x,0) = (x,2)$. Again, it is easy to show that the foliation can be defined in $\TT^3 = \TT^2 \times [0,2]/_{\varphi_2}$.

By the previous Theorem, we know that the resulting foliation is almost aligned with a linear foliation of $\TT^3$. Since we have assumed that there is no torus leaves of $\cW$ we know that this linear foliation cannot be the one given by the planes $\RR^2 \times \{t\}$ so, it must be transverse to the boundaries of $\TT^2 \times [0,1]$. This concludes the proof. 
\end{proof}

\begin{remark} As a consequence of the previous result we get the following: The foliations $\cW\cap (\TT^2 \times \{0\})$ and $\cW \cap (\TT^2 \times \{1\})$ are almost aligned with each other. In particular, one can prove that if in one of the boundary components is almost parallel to a linear foliation, then the whole foliation $\cW$ is almost parallel to a linear foliation of $\TT^2$ times $[0,1]$. 
\end{remark}

Consider the manifold $M_\psi$ obtained by $\TT^2 \times [0,1]$ by identifying $\TT^2 \times \{0\}$ with $\TT^2 \times \{1\}$ by a diffeomorphism $\psi$. Let $p: M_\psi \to S_1=[0,1]/_\sim$  given by the projection in the second coordinate. Any torus bundle over the circle can be constructed this way, naturally, if $\psi$ and $\psi'$ are isotopic then $M_\psi$ and $M_{\psi'}$ are diffeomorphic. 

The construction of $M_\psi$ determines a incompressible torus in $M_\psi$ which we will assume remains fixed. Under this choice of incompressible torus we can consider a family of foliations of $M_\psi$ transverse to such torus. 

We are now able to classify foliations in torus bundles over $S^1$ depending on the isotopy class of $\psi$. 

Any manifold of the form $\cH/\Gamma$ can be constructed as a torus bundle over $S^1$: The monodromy being given by (something isotopic to) the diffeomorphism $\psi_k: \RR^2/\ZZ^2 \to \RR^2/\ZZ^2$ given by:

$$  \psi_k( x) = \left(%
\begin{array}{cc}
  1 & k \\
  0 & 1 \\
\end{array}%
\right) x  \qquad  \mod \ZZ^2  $$

In the case that $\psi: \TT^2 \cong S^1 \times S^1 \to \TT^2$ is a Dehn-twist of the form:

$\psi(t,s)= (t, s+kt) (\mod \ZZ^2)$

\noindent $M_\psi$ is homeomorphic to a nilmanifold of the form $\cH/\Gamma$. We define the foliations $\cF_\theta$ on $M_\psi$ given by starting with the linear foliation $\cL$ of $\TT^2$ by circles of the form $\{t\} \times S^1$ and we consider the foliation $\cL \times [0,1]$ of $\TT^2 \times [0,1]$. The foliation $\cF_\theta$ will be the foliation obtained by gluing $\TT^2 \times \{0\}$ with $\TT^2 \times \{1\}$ by the diffeomorphism 

$$ \psi_\theta: \TT^2 \times \{0\} \to \TT^2 \times \{1\} \qquad \psi_\theta (t,s,0) = (t + \theta, s+ kt, 1) $$ 

\begin{remark} Notice that if $\cW$ is a foliation of $M_\psi$ which is transverse to $T$ the torus obtained by projection of $\TT^2 \times \{0\} \sim \TT^2 \times \{1\}$ we know that it must be invariant by a map of $T$ which is isotopic to $\psi$. 
\end{remark}

The foliation $\cF_\infty$ is the foliation by the fibers of the torus bundle. These foliations correspond to the foliations of the form $\cF_\pi$ in $\cH/\Gamma$ defined in section \ref{SectionNilm}. In particular, we know that they are pairwise not almost parallel.

\begin{thm}\label{Teo-ClassificationNilmanifolds} Let $\cW$ be a
codimension one Reebless foliation of $M_\psi$ where $\psi$ is a Dehn twist as above.
Then, $\cW$ is almost aligned with $\cF_\theta$ for some $\theta
\in \RR \cup \{\infty\}$. Moreover, if $\theta$ is irrational then $\cW$ is almost parallel to $\cF_\theta$. 
\end{thm}

\begin{proof} If $\cW$ has a torus leaf, this torus must be incompressible by Novikov's Theorem (\cite{Solodov, CandelConlon}). We can cut the foliation along this torus. By doing the same doubling procedure as in Proposition \ref{Prop-FolT2int} we obtain a foliation of $\TT^3$ and using Theorem \ref{FoliationT3} we deduce that $\cW$ is almost aligned with a foliation of the form $\cF_\theta$ with $\theta$ being irrational. 

If $\cW$ has no torus leaves, we can consider the torus $\TT^2 \times \{0\} \en M_\psi$ which is incompressible. Using Theorem \ref{Teo-Gabai} we can make an isotopy and assume that the foliation $\cW$ is transverse to this torus (recall from Proposition \ref{Prop-PropertiesAlmostParallel} that the isotopy does not affect the equivalence class of the foliation under the relation of being almost parallel). Here we are using the fact that the isotopy of the torus can be extended to a global isotopy of $M$ (see for example Theorem 8.1.3 of \cite{Hirsch}). 

We can cut $M_\psi$ by this torus and apply Proposition \ref{Prop-FolT2int} to obtain that $\cW$ in $\TT^2 \times [0,1]$ is almost aligned to a linear foliation of $\TT^2$ times $[0,1]$. In fact, if the foliation is not almost parallel to the linear foliation we deduce that the foliation in $\TT^2 \times \{0\}$ must have Reeb annuli (see section 4.A of \cite{Pot2}). Since the foliation in $\TT^2 \times \{0\}$ must be glued by $\psi$ with the foliation in $\TT^2 \times \{1\}$ we deduce that it must permute these Reeb annuli which are finitely many. So, we get that there is a periodic circle of the foliation of $\TT^2 \times \{0\}$ by $\psi$ which implies the existence of a torus leaf for $\cW$. We deduce that $\cW$ in $\TT^2 \times [0,1]$ is almost parallel to a linear foliation of $\TT^2$ times $[0,1]$.

Now, we must show that this linear foliation corresponds to the linear foliation $\cL$ by circles of the form $\{t\} \times S^1$ but this follows from the fact that the foliation is invariant by $\psi$. 

Now we must see that after gluing the foliation is almost parallel to some $\cF_\theta$. This follows from the following fact, since in the boundary it is almost parallel to the foliation $\cL$, we know that it has at least one circle leaf $L$. Now we obtain the value of $\theta$ by regarding the relative order of the images of $\psi^n(L)$ and performing a classical rotation number argument as in the circle. 
\end{proof}

When $\psi$ is isotopic to Anosov, the classification gives only three possibilities. 

We consider then $A$ a hyperbolic matrix in $SL(2,\ZZ)$ and in $M_A$ we consider the following linear foliations: $\cF^{cs}_A$ is given by the linear foliation which is the projection of $\cL^s \times [0,1]$ where $\cL^s$ is the linear foliation corresponding to the strong stable foliation of $A$ and similarly we obtain $\cF^{cu}_A$ as the projection of $\cL^u \times [0,1]$ where $\cL^u$ is the linear foliation which corresponds to the strong unstable foliation. 

Finally, we consider the foliation $\cF_T$ which is the projection of foliation by tori $\TT^2 \times \{t\}$ to $M_A$. Clearly, these 3 foliations are pairwise not almost parallel to each other. 

\begin{thm} Let $\cW$ be a Reebless foliation of $M_A$, then, $\cW$ is almost aligned to one of the foliations $\cF^{cs}_A, \cF^{cu}_A$ or $\cF_T$. Moreover, if $\cW$ has no torus leaves, then $\cW$ is almost parallel to either $\cF^{cs}_A, \cF^{cu}_A$ and if it is almost aligned with $\cF_T$ it has a torus leaf. 
\end{thm}

\begin{proof} The first part of the proof is as in the previous Theorem: If $\cW$ has a torus leaf, it must be isotopic to $T$ the projection of $\TT^2 \times \{0\}$ since it is the only incompressible embedded torus in $M_A$ and we get that we get that $\cW$ is almost parallel to $\cF_T$. 

Otherwise, we can assume that $\cW$ is transverse to $T$ and we obtain a foliation of $\TT^2 \times [0,1]$ which is almost aligned with a foliation of the form $\cL \times [0,1]$ and which in $T$ is invariant under a diffeomorphism $f$ isotopic to $A$.

This implies that the linear foliation $\cL$ is either the strong stable or the strong unstable foliation for $A$, and in particular, since it has no circle leaves, we get that $\cW$ in  $\TT^2 \times [0,1]$ is almost parallel to $\cL \times [0,1]$. 

Now, since the gluing map $f$ is isotopic to $A$, we know it is semiconjugated to it, so, we get that after gluing, the foliations remain almost parallel. 
\end{proof}



\end{document}